\nonstopmode
\documentclass[]{amsart}
\usepackage{amssymb}
\usepackage[hyperindex=true,bookmarks=true,bookmarksnumbered=true]{hyperref}
\usepackage{graphicx}
\usepackage{epsfig}
\usepackage[all]{xy}
\allowdisplaybreaks

\def\undersetbrace#1\to#2{\underbrace{#2}_{#1}}
\def\oversetbrace#1\to#2{\overbrace{#2}^{#1}}
\def\AMSunderset#1\to#2{\underset{#1}{#2}}
\def\AMSoverset#1\to#2{\overset{#1}{#2}}

\def\East#1#2{-\raisebox{0.1pt}{$\mkern-16mu\frac{\;\;#1\;}{\;\;#2\;}\mkern-16mu$}\to}

\newcommand{\nmb}[2]{\ifx!#1{\ref{nmb:#2}}%
\else\if.#1{\label{nmb:#2}}%
\else\if0#1{\label{nmb:#2}}%
\else{{#2}}%
\fi\fi\fi}
\swapnumbers
\newtheorem{proposition}[subsection]{Proposition}
\newtheorem*{proposition*}{Proposition}
\newtheorem{theorem}[subsection]{Theorem}
\newtheorem*{theorem*}{Theorem}
\newtheorem{lemma}[subsection]{Lemma}
\newtheorem*{lemma*}{Lemma}

\newtheorem*{corollary*}{Corollary}

\def\ign#1{}             
\def\o{\circ}
\def\X{\mathfrak X}
\def\al{\alpha}
\def\be{\beta}

\def\io{\iota}

\def\si{\sigma}

\def\ph{\varphi}

\def\ps{\psi}
\def\om{\omega}
\def\Ga{\Gamma}

\def\Om{\Omega}

\def\x{\times}
\def\p{\partial}
\let\on=\operatorname
\def\L{\mathcal L}
\def\Fl{\operatorname{Fl}}
\usepackage{accents}
\newlength{\dhatheight}

\begin{document}
\title[]
{The Schouten-Nijenhuis bracket in infinite dimensions
}
\author{Peter W. Michor}
\address{
Peter W. Michor:
Fakult\"at f\"ur Mathematik, Universit\"at Wien,
Nordbergstrasse 15, A-1090 Wien, Austria
}
\email{Peter.Michor@univie.ac.at}
\date{\today}

\keywords{Schouten-Nijenhuis bracket, Infinite dimensional smooth convenient manifolds}
\subjclass{58B20, 37K06}
\begin{abstract} 
The Schouten-Nijenhuis bracket on  smooth infinite-dimensional manifolds $M$ is developed  in two steps: For summable multivector fields whose pointwise dual are all differential form, and in an extended form for multivector fields which are sections of   
$L^{\bullet}_{\text{skew}}(T^*M,\mathbb R)$. We need to either assume that $C^{\infty}(M)$ separates points on $TM$, or consider sheaves of local sections; see \nmb!{1.1}.
\end{abstract}
\def\LaTeXonly{}

\maketitle

\section*{Introduction}

We develop the Schouten-Nijenhuis bracket for skew multivector fields on infinite dimensional 
smooth manifolds. The initial point of view in Sections \nmb!{1} and \nmb!{2} is that the space of multivector fields is predual to the 
only space of differential forms admitting exterior derivative, general pullbacks, and insertion operators; 
see \cite[33.21]{KrieglMichor97}. We use the bornological tensor product for that. 
We fist give a direct definition of the Schouten-Nijenhuis bracket for summable multivector fields (see \nmb!{1.4}) whose 
pointwise dual consists of differential forms in Section \nmb!{1}. 
Then we use the duality between summable multivector fields and differential forms to derive 
another formula for the Schouten-Nijenhuis bracket \thetag{\nmb!{2.3}.7} which in finite dimensions and using slightly different conventions is due to Tulczyjew \cite{Tulczyjew74}. In Section \nmb!{3} we turn this around and use Tulczyjew's formula to extend the Schouten-Nijenhuis bracket for multivector fields which are sections of the bundle $L^{\bullet}_{\text{skew}}(T^*M,\mathbb R)$.

The main use of the Schouten-Nijenhuis bracket is to recognise Poisson structures (see \nmb!{3.6}) and their formal deformations. The general Schouten-Nijenhuis bracket for general multivector fields as presented in Section \nmb!{3} is needed because there exist quite general Poisson structure in infinite dimensions; see \cite{OdzijewiczRatiu03}, \cite{OdzijewiczRatiu03a}, \cite{Tumpach20} and the overview in \cite{GRT24}. In \cite{BGT18} one finds even `queer' Poisson structures where the bracket is a bidifferential operator of order higher than one. One can adapt the Schouten-Nijenhuis bracket as developed in Section \nmb!{3} to catch also this situation by adapting the bundle of differential forms accordingly choosing from \cite[33.21 and 32.5]{KrieglMichor97}. This is not done here since there are no serious applications of such Poisson structures available yet. Moreover, many Poisson structures are only defined on duals of subbundles of $T^*M$. The results of this paper can maybe adapted to these situations. 

Weak symplectic structures in infinite dimensions in the convenient setting and their Poisson brackets have been treated in 
\cite[Section 48]{KrieglMichor97}. Poissons structures in infinite dimensionas are treated in 
\cite{NST14}, \cite{CabauPelletier24b}, and an overview is given by \cite{GRT24}. 
This paper finally aswers a question posed by Tudor Ratiu in 2013. 
I thank Alice Barbora Tumpach and Praful Rahangdale for discussions and hints. The work on this paper was supported by the thematic program "Infinite-dimensional Geometry: Theory and Applications", Jan.13 -- Feb.14, at the Erwin Schr\"odinger Institute in Vienna.

\section{The Schouten bracket for summable skew multivector fields}\nmb0{1}

\subsection{\nmb.{1.1}What manifolds do we use?}
We shall use smooth ($C^{\infty}$) calculus as described in \cite{KrieglMichor97}.
Let $M$ be a smooth manifold modelled on a convenient vector space $E$. Thus we have a smooth atlas 
$M\supset U_\al\East{u_\al}{} u_\al(U\al)\subset E$ and all chart changings 
$u_{\al\be}:u_\be(U_\al\cap U_\be)\to u_\al(U_\al\cap U_\be)$ are smooth.
Let $\pi_M:TM\to M$ be its (kinematic) tangent bundle which is described by induced charts
$TM\supset TU_\al \East{Tu_\al} u(U_\al)\x E\subset E\x E$.
The Lie algebra of vector fields $\X(M)$ is the space of sections of $TM\to M$.
We shall use the graded differential algebra of differential forms consisting of smooth sections of the bundle of bounded skew symmetric multilinear forms $L_{\text{skew}}^*(TM,\mathbb R)$ on the the tangent bundle, see \cite[Section 33]{KrieglMichor97}: 
$$\Om(M) = \bigoplus_{k=0}^\infty \Om^k(M) = \bigoplus_{k=0}^\infty C^{\infty}(M \leftarrow L^k_{\text{skew}}(TM,\mathbb R)).$$
This algebra admits all desired operations $\L_X, d, f^*$; see \cite[33.21]{KrieglMichor97}.

In this paper we consider only manifolds $M$ having the following property:
\emph{For each covector $\al\in T^*M$ there exists a function $f\in C^{\infty}(M)$ with $df_{\pi(\al)}=\al$}.
The following classes of manifolds have this property: 
\begin{itemize}
\item Smoothly paracompact manifolds (having smoothly paracompact modelling spaces), see \cite[Section 33]{KrieglMichor97}. 
\item Each manifold $M$ such that $C^{\infty}(M,\mathbb R)$ separates points on $TM$. 
\item Manifolds smoothly immersed into a convenient vector space, i.e., admitting a smooth mapping $F:M\to E$ such that $dF_x:T_xM\to E$ is injective for each $x\in M$; then bounded linear functionals on $E$ do the job.  
\end{itemize}
One can drop this property if one considers sheaves of local sections of all relevant bundles below. 
For simplicity's sake we stick with global sections.

\subsection{\nmb.{1.2}The completed bornological tensor product}
For a convenient vector space $E$, let 
$E\bar\otimes_\be E$ be the $c^\infty$-completed (see \cite[4.29]{KrieglMichor97}) 
bornological tensor product which linearizes bibounded bilinear 
mappings; see \cite[5.7]{KrieglMichor97}. 
If $E$ is a Banach or Fr\'echet or (DF) space then each bibounded bilinear mapping is jointly 
continuous and thus $E\bar\otimes_\be E$ agrees with the completed projective tensor product of 
Grothendieck \cite{Grothendieck55}; see \cite[5.8]{KrieglMichor97}.

Let $\bigwedge^nE$ be the (Mackey-) closed linear subspace of all {\em alternating
tensors} in $\bar\bigotimes^n_\be E$, see \cite[Section 5]{KrieglMichor97}.
It is the universal solution for convenient vector spaces $F$ of the linearization problem
$L(\bigwedge^n E,F)\cong L_{\text{alt}}^n(E;F)$,
where $L^n_{\text{alt}}(E;F)$ is the space of all bounded $n$-linear alternating mappings 
$E\x\dots\x E\to F$, a direct summand of $L^n(E;F):=L(E,\dots,E;F)$.
By \cite[5.9.5]{KrieglMichor97} the mapping 
$\bigwedge^n : L(E,F)\to L(\bigwedge^nE,\bigwedge^nF)$ is bounded multilinear and thus smooth.

\subsection{\nmb.{1.3}Summable multivector fields}

We apply the smooth mapping 
$$\bigwedge^n : L(E,F)\to L(\bigwedge^nE,\bigwedge^nF)$$
to the chart change mappings for the tangent bundle $TM\to M$ to obtain the 
smooth vector bundle  $\pi_M:\bigwedge^n TM \to M$ of {\em summable} $n$-multivectors on $M$. 
Note that the space linearly generated by $X_1\wedge \dots\wedge X_n$ for $X_i\in T_xM$ is dense in 
the fiber $\bigwedge^n T_xM$.
The space $\Ga(\bigwedge^n TM)$ of smooth sections of this bundle is the space of {\em summable 
multivector fields} on $M$. We write $\Ga(\bigwedge^0 TM)=C^\infty(M,\mathbb R)$ and  
$\Ga(\bigwedge TM) = \bigoplus_{n\ge 0} \Ga(\bigwedge^n TM)$ which is a graded  commutative algebra for the usual 
wedge-product (see \cite[33.8]{KrieglMichor97} for the convention) of multivector fields for the grading $(\Ga(\bigwedge TM),\;\wedge \;)_n = 
\Ga(\bigwedge^n TM)$. The wedge product is a bounded bilinear operation on the convenient space 
$\Ga(\bigwedge TM)$, by the universal property of the the bornological tensor product.  

\begin{theorem}\nmb.{1.4} Let $M$ be a smooth manifold modeled on a convenient vector space $M$. 
Then the following bracket (where $X_i$ and $Y_j$ are vectorfields) 
is well defined and bounded and extends to $\Ga(\bigwedge TM)$. 
\begin{align*}
[1,U] &= 0\qquad\text{  for all }U\in\Ga(\bigwedge TM)
\\
[f,U] &= -\bar\io(df)U\qquad\text{  for }f\in C^\infty(M)\text{  and }U\in\Ga(\bigwedge TM)
\\&\text{ where }\quad
\bar\io(df)(X_1\wedge \dots\wedge X_k) = \sum_i(-1)^{i-1}df(X_i)\cdot X_1\wedge\cdots \widehat{X_i}\dots\wedge X_k
\\
[X_1\wedge &\dots\wedge X_k,Y_1\wedge \dots\wedge Y_l] 
=\\&
=\sum_{i,j}(-1)^{k-i+j-1} 
X_1\wedge\cdots \widehat{X_i}\dots\wedge X_k\wedge [X_i,Y_j]\wedge Y_1\wedge \cdots\widehat{Y_j}\dots\wedge Y_l
\\&
=\sum_{i,j}(-1)^{i+j} 
[X_i,Y_j]\wedge X_1\wedge\cdots \widehat{X_i}\dots\wedge X_k\wedge Y_1\wedge \cdots\widehat{Y_j}\dots\wedge Y_l
\end{align*}
For the grading $(\Ga(\bigwedge TM),[\;,\;])_n = \Ga(\bigwedge^{n+1} TM)$ we get 
a convenient graded Lie algebra which is compatible with the wedge product. Namely, for 
$U\in\Ga(\bigwedge\nolimits^{u} TM)$ and $ V\in\Ga(\bigwedge\nolimits^{v} TM)$ we have:
\begin{align*}
[U,V] &= -(-1)^{(u-1)(v-1)}[V,U]
\\
[U,[V,W]] &= [[U,V],W] + (-1)^{(u-1)(v-1)} [V,[U,W]]
\\
[U,V\wedge W] &= [U,V]\wedge W + (-1)^{(u-1)v}V\wedge [U,W]\,.
\end{align*}
\end{theorem}
This is the infinite dimensional version of the Schouten-Nijenhuis bracket, which was found by 
\cite{Schouten40}; that it satisfies the graded Jacobi identity is due to \cite{Nijenhuis55}.
The approach given here is the infinite dimensional version of \cite[Theorem 1.2]{Michor87a}. 
\\
\emph{Note on conventions.} The reason for the convention (which agrees with the one used by Koszul \cite{Koszul85}) used in  Theorem \nmb!{1.4} and later is as follows: It fits into the following universal property:
For any Lie algebra $\mathfrak g$ and $\mathbb N_{\ge0}$-graded Lie algebra $A$ any degree 0 Lie algebra homomorphism $\ph$ extends uniquely to a homomorphism of graded Lie algebas  $\tilde\ph$ as in the diagram below, where one should shift the degree by -1 in the top row:
$$
\xymatrix{
\bigwedge\nolimits^1\mathfrak g\; \ar@{^(->}[r] & \bigwedge\nolimits^{\bullet}\mathfrak g\\
\mathfrak g \ar[u]^{\cong}  \ar[r]^{\ph\qquad}  & A=\bigoplus\nolimits_{k=0}^\infty A^k                  \ar@{-->}[u]_{\tilde\ph} 
}
$$
A 
different approach is by Tulczyjew \cite{Tulczyjew74}; We shall use it below to extend the bracket to far larger spaces of multivector fields.
Other conventions are as follows: 
\begin{itemize}
\item Tulczyjew  \cite{Tulczyjew74} uses $-(-1)^{(u-1)(v-1)}[U,V] = -[V,U]$. This version has a more natural relation to the Lie derivative than \thetag{\nmb!{2.3}.6}.
\item Vaisman \cite{Vaisman94} and Lichnerowicz \cite{Lichnerowicz77} use $(-1)^{u-1}[U,V] = (-1)^{uv}(-1)^{v-1}[V,U]$; this choice shifts the signs for the graded skew symmetry as indicated and for the graded Jacobi identity. 
\end{itemize}

\begin{proof}
For $f\in C^\infty(M,\mathbb R)$ it is easily checked that  
\begin{multline*}
[X_1\wedge \dots\wedge X_k,Y_1\wedge \dots \wedge f.Y_j\wedge \dots\wedge Y_l]
= f.[X_1\wedge \dots\wedge X_k,Y_1\wedge \dots\wedge Y_l] 
+\\ 
+(-1)^{k-1}\bar\io(df)(X_1\wedge \dots\wedge X_k)\wedge Y_1\wedge \dots\wedge Y_l\,.
\end{multline*}
Thus the bracket, which a priori is defined on (the dense algebraic tensor products)
$\bigwedge^k \Ga(TM) \x \bigwedge^l\Ga(TM) \to \bigwedge^{k+l-1}\Ga(TM)$ factors to
$$
\bigwedge\nolimits^k_{C^\infty(M)} \Ga(TM) \x \bigwedge\nolimits^l_{C^\infty(M)}\Ga(TM) 
\to \bigwedge\nolimits^{k+l-1}_{C^\infty(M)}\Ga(TM)
$$
which equals 
$$
\Ga(\bigwedge\nolimits^k TM) \x \Ga(\bigwedge\nolimits^l TM)\to \Ga(\bigwedge\nolimits^{k+l-1} TM)\,.
$$
So it is well defined, at least for smoothly normal manifolds. For others one has to use the same 
argument on a local chart; this is done below in \thetag{\nmb!{3.4}.4} in a more general setting.
Since the bracket is clearly bounded, it extends to the Mackey closure.
Finally we have to check the graded Jacobi identity. This is an elementary but tedious computation. 
The graded derivation property with respect to the wedge product is easily checked. 
\end{proof}

\section{Using the duality between summable skew multivector fields and differential forms}\nmb0{2}

Having established the Schouten-Nijenhuis bracket for summable skew multivector fields in Section \nmb!{1}, we now 
make full use use of the duality 
between summable multivector fields and differential forms. 

\subsection{The duality between multivector fields and differential forms}\nmb.{2.1}
Let $M$ be a smooth manifold modeled on a convenient vector space $E$.
By the universal property of the bornological tensor product described in \nmb!{1.2}, the dual 
space of $\bigwedge^n E$ is the space $L^n_{\text{skew}}(E;\mathbb R)$. 
Using and extending the conventions  of \cite[5.30]{Greub78}, see also  \cite[Section 33]{KrieglMichor97}, we start from the duality 
\begin{gather*}
\langle\;,\;\rangle: \bigwedge\nolimits^n E^* \x \bigwedge\nolimits^n E \to \mathbb R
\\
\langle \ph_1\wedge \dots\wedge \ph_n, X_1\wedge \dots\wedge X_n\rangle = \det(\langle\ph_i,X_j\rangle_{i,j}) 
\end{gather*}
which is $1/n!$ of the restriction of the duality $\langle\;,\;\rangle: \otimes^n E^* \x \otimes^n E \to \mathbb R$,
we get the complete fiberwise duality  
\begin{gather*}
\langle \;,\; \rangle:\Om^n(M)\x \Ga(\bigwedge\nolimits^n TM)\to C^\infty(M)\,,
\\
\langle \om,X_1\wedge \dots\wedge X_n\rangle = \om(X_1,\dots, X_n)
\end{gather*}
We have the following dual pairs of  operators: For $\om\in\Om^p(M)$ the linear map
$\mu(\om):\Om^k(M)\to \Om^{k+p}(M)$ given by $\mu(\om)\ps := \om \wedge \ps$ is the fiberwise dual 
operator to 
$\bar\io(\om):\Ga(\bigwedge^{k+p} TM)\to \Ga(\bigwedge^k TM)$, where
\begin{multline*}
\bar\io(\om)(X_1\wedge \dots\wedge  X_{k+p}) 
= \\
=\frac1{p!k!}\sum_{\si\in \mathfrak S_{k+p}} \on{sign(\si)} \om(X_{\si(1)},\dots,X_{\si(p)}) 
X_{\si(p+1)}\wedge \dots\wedge X_{\si(p+k)}\,.
\end{multline*}
since for $\ph\in \Om^k(M)$ we have
\begin{align*}
\langle \ph&,\bar\io(\om)(X_1\wedge \dots\wedge  X_{k+p})\rangle = 
\langle\om\wedge \ph, X_1\wedge \dots\wedge  X_{k+p}\rangle
= (\om\wedge \ph) (X_1, \dots,X_{k+p})
\\&
= \frac1{p!k!}\sum_{\si\in \mathfrak S_{k+p}} \on{sign(\si)} \om(X_{\si(1)},\dots,X_{\si(p)})\ph(X_{\si(p+1)}, \dots,X_{\si(k+p)})\,.
\end{align*}
Likewise, for $U\in\Ga(\bigwedge^p TM)$ the fiberwise linear mapping $\bar\mu(U):\Ga(\bigwedge^k 
TM)\to \Ga(\bigwedge^{k+p} TM)$ given by $\bar\mu(U)V=U\wedge V$ is the fiberwise dual of the `insertion operator'
$i(U):\Om^{k+p}(M)\to \Om^k(M)$.

\begin{lemma}\nmb.{2.2}
Let $U$ be in $\Ga(\bigwedge^u TM)$. Then we have:
\begin{enumerate}
\item 
$i(U):\Om(M)\to \Om(M)$ is a homogeneous bounded module 
homomorphism of degree $-u$.  It is a graded derivation of $\Om(M)$ if and only if $p=1$. For $f\in C^{\infty}(M)$ we have $i(f)\om=f.\om$.
\item 
$i(U\wedge V)=i(V)\o i(U)$, thus the graded commutator vanishes: 
$$
[i(U),i(V)] = i(U)i(V) - (-1)^{uv}i(V)i(U)=0.
$$
\item
$i(U)(\om\wedge \ps) = i(\bar\io(\om)U)\ps + (-1)^u\om\wedge i(U)\ps$ for $\om\in\Om^1(M)$ and $\ps\in\Om(M)$.
\end{enumerate}
\end{lemma}

\begin{proof}
\thetag{1} follows from the definition. 
\thetag{2} is the dual of $\bar \mu (U\wedge V)W= U\wedge V\wedge W = \bar\mu(U)\bar\mu(V)W$.
\thetag{3} For vector fields $X_j$ and decomposable $U=X_1\wedge\dots\wedge  X_u$ we have
\begin{align*}
&i(X_1\wedge\dots\wedge  X_u)\mu(\om) = i(X_u)\dots i(X_1)\mu(\om)  
\\&
= i(X_u)\dots i(X_2)\big(-\mu(\om)i(X_1) + \mu(i(X_1)\om) \big)
\\&
=  \om(X_1).i(X_u)\dots i(X_2) -  i(X_u)\dots i(X_3)\big(\mu(\om)i(X_2) + \mu(i(X_2)\om) \big)
\\&
= \sum_{j=1}^u (-1)^{j-1} \om(X_j).i(X_u)\dots\widehat{i(X_j)}\dots i(X_1) +(-1)^u \mu(\om)  i(X_u)\dots i(X_1)
\\&
= i\Big(\sum_{j=1}^u (-1)^{j-1} X_u\wedge \dots\wedge \bar\io(\om)X_j\wedge \dots\wedge  X_1\Big) +(-1)^u \mu(\om)  i(X_1\wedge \dots\wedge  X_u)
\\&\implies [i(U),\mu(\om)] = i(\bar\io(\om)U)  \qedhere
\end{align*}
\end{proof}

\subsection{The Lie differential operator}\nmb.{2.3}
For $U\in \Ga(\bigwedge^u TM)$ we define the {\em Lie derivation} $\L(U):\Om^k(M)\to \Om^{k-u+1}(M)$ by
\begin{equation*}
\L(U) := [i(U),d] = i(U)\o d - (-1)^p d\o i(U) \tag{1}
\end{equation*}

which is homogeneous of degree $1-u$ and is called the Lie differential operator. It is a 
derivation if and only if $U$ is a vector field. We have $[\L(U),d]=0$ by the graded Jacobi 
identity of the graded commutator.

\begin{theorem*}
Let $U\in\Ga(\bigwedge^u TM)$, $V\in \Ga(\bigwedge^v TM)$, and $f\in C^\infty(M)$.
Then we have:
\begin{align*}
\L(U\wedge V) &= i(V)\o \L(U) + (-1)^u\L(V)\o i(U)
\tag{2}\\
\L(X_1\wedge \dots\wedge X_u) &= \sum_j 
(-1)^{j-1}i(X_u)\cdots i(X_{j+1})\L(X_j)i(X_{j-1})\cdots i(X_1)
\tag{3}\\
\L(f) &= [i(f),d] = [\mu(f),d] = - \mu(df)
\tag{4}\\
[\L(U),i(V)] &= (-1)^{(u-1)(v-1)} i([U,V]) = -i([V,U])
\tag{5}\\
[\L(U),\L(V)] &= (-1)^{(u-1)(v-1)} \L([U,V]) = -\L([V,U])
\tag{6}\\
\langle d\om,-[V,U] \rangle &= \langle di(V)d\om,U \rangle -(-1)^{(u-1)(v-1)} \langle di(U)d\om,V\rangle
\tag{7}\end{align*}
\end{theorem*}

Formula \thetag{7} suffices to compute $[U,V]$ in a local chart, and it remains valid if we insert 
any closed form instead of $d\om$ since \thetag{7} is a local formula, and locally any closed form is exact by the Poincar\'e lemma \cite[33.20]{KrieglMichor97}. Formula \thetag{7}  was the starting point of the treatment of the Schouten-Nijenhuis bracket in  \cite{Tulczyjew74}.

\begin{proof}
\begin{align*}
\L(U\wedge V) &= [i(U\wedge V),d] = [i(V)i(U),d] = i(V)i(U)d -(-1)^{u+v}di(V)i(U) \tag{2}
\\& 
= i(V)\big(i(U)d - (-1)^udi(U)\big) + (-1)^u\big(i(V)d - (-1)^vdi(V)\big)i(U)
\\&
= i(V)\o \L(U) + (-1)^u\L(V)\o i(U)
\end{align*}
\thetag{3} Use induction on $u$, using \thetag{2}:
\begin{align*}
&\L(X_1\wedge \dots\wedge X_u) = i(X_u)\L(X_1\wedge \dots\wedge X_{u-1}) - (-1)^u \L(X_u)i(X_1\wedge \dots\wedge X_{u-1})
\\&
= \sum_{j=1}^{u-1} (-1)^{j-1} i(X_u)\dots \L(X_j)\dots i(X_1)  + (-1)^{u-1} \L(X_u)i(X_{u-1}) \dots i(X_1)
\\&
= \sum_{j=1}^u (-1)^{j-1} i(X_u)\dots \L(X_j)\dots i(X_1)  
\end{align*}
\begin{align*}
& d(f.\om) = df\wedge \om + f.d\om \quad\implies\quad d\o \mu(f) = \mu(df) + \mu(f) \o d  \quad\implies\tag{4}
\\&
\L(f) = [i(f),d] = [\mu(f),d] = - \mu(df)
\end{align*}
\thetag{5} 
We start with
\begin{align*}
[\L(U),i(V)] &= [[i(U),d],i(V)] = [i(U),[d,i(V)]] -(-1)^u [d,[i(U),i(V)]] \tag{a}   
\\&
= (-1)^{v-1}[i(U),\L(V)] + 0.
\end{align*}
We use induction on $u+v$. For $u+v=0$ we have 
$$[\L(f),i(g)]=[-\mu(df),\mu(g)]=0 = -i([f,g]).$$
For $u+v=1$, by \thetag{a} it suffices to check 
\begin{align*}
[\L(X),i(f)] &= \L(X)\mu(f) - \mu(f)\L(X)= i(df(X)).
\end{align*}
By \thetag{a} again, for the induction step it suffices to check
\begin{align*}
[\L(X\wedge U),i(V)] &= [i(U)\L(X) -\L(U)i(X),i(V)]\quad\text{ by \thetag{2}}
\\&
= i(U)\L(X)i(V) -(-1)^{uv} i(V)i(U)\L(X) 
\\&\qquad
- \L(U)i(X)i(V) + (-1)^{uv} i(V)\L(U)i(X)
\\&
= i(U)[\L(X),i(V)] -(-1)^{v} [\L(U),i(V)]i(X) 
\\&
=i(U)i(-[V,X]) -(-1)^v i(-[V,U])i(X)\quad\text{ by induction }
\\&
=-i\big([V,X]\wedge U + (-1)^{v-1} X\wedge [V,U] \big)\quad\text{ by \thetag{\nmb!{2.2}.2} }
\\&
= -i([V,X\wedge U])\quad\text{ by Theorem \nmb!{1.4}. }
\end{align*}
\thetag{6}
We use again induction on $u+v$.  For $u+v=0$ we have by \thetag{4}  
$$[\L(f),\L(g)] = [-\mu(df),-\mu(dg)] = \mu(df)\mu(dg)+\mu(dg)\mu(df) = 0 = \L(-[g,f]).$$
For $u+v=1$ we have 
\begin{align*}
[\L(X),\L(f)] &= -[\L(X),\mu(df)] = - \mu(\L(X)df) 
\\
\L(-[f,X]) &= \L(\bar\io(df)X) = \L(df(X)) = -\mu(d(df(X))) = -\mu(\L(X)df).
\end{align*}
The induction step:
\begin{align*}
[\L(&X\wedge U),\L(V)] = [i(U)\L(X)-\L(U)i(X),\L(V)] \quad\text{ by \thetag{1}}
\\&
= i(U)\L(X)\L(V) -(-1)^{u(v-1)} \L(V)i(U)\L(X) 
\\&\qquad
- \L(U)i(X)\L(V) + (-1)^{u(v-1)} \L(V)\L(U)i(X)
\\&
=  i(U)[\L(X),\L(V)] -(-1)^{u(v-1)} [\L(V),i(U)]\L(X) 
\\&\qquad
- \L(U)[i(X),\L(V)] + (-1)^{u(v-1)} [\L(V),\L(U])i(X)
\\&
=  i(U)\L(-[V,X]) +(-1)^{u(v-1)} i([U,V])\L(X) \quad\text{ by induction and \thetag{4}} 
\\&\qquad
+ (-1)^{(v-1)}\L(U)i([X,V]) + (-1)^{u(v-1)} \L([U,V])i(X)
\\
\L(&-[V,X\wedge U]) = \L(-[V,X]\wedge U -(-1)^{v-1} X\wedge [V,U]) \quad\text{ by Theorem \nmb!{1.4}}
\\&
=  i(U)\L(-[V,X]) - (-1)^{v}\L(U)i([X,V])  \quad\text{ by  \thetag{2}} 
\\&\qquad
-(-1)^{v} i([V,U])\L(X) + (-1)^{u(v-1)} \L([U,V])i(X)\quad \text{ which proves \thetag{6}.}
\end{align*}
\thetag{7} 
For $\om\in\Om^{u+v-3}(M)$ we have $i(-[V,U])\om=0$ by degree, thus
\begin{align*}
\L(-[V,U])\om &= i(-[V,U])d\om + 0 = \langle d\om, -[V,U]\rangle
\\
\L(U)\L(V)\om &= [i(U),d]i[(V),d]\om  =  i(U)di(V)d\om + 0 + 0 + 0 = \langle di(V)d\om,U,\rangle
\\
\L(V)\L(U)\om &= \langle di(U)d\om,V\rangle
\\
\L(-[V,U])\om &= [\L(U),\L(V)]\om = \langle di(V)d\om,U,\rangle -(-1)^{(u-1)(v-1)}  \langle di(U)d\om,V\rangle \qedhere
\end{align*}
\end{proof}

\subsection{Naturality of the Schouten Nijenhuis bracket}\nmb.{2.4}
Let $f:M\to N$ be a smooth mapping between convenient manifolds. We say that 
$U\in \Ga(\bigwedge^u TM)$ and $U'\in \Ga(\bigwedge^u TN)$ are $f$-related, if 
$\bigwedge^uTf.U = U'\o f$ holds:
$$
\xymatrix{
\bigwedge\nolimits^uTM \ar[r]^{\bigwedge\nolimits^u Tf} & \bigwedge\nolimits^uTN \\
M \ar[u]^{U} \ar[r]^{f} & N \ar[u]_{U'}
}
$$

\begin{proposition}\nmb.{2.5}
\begin{enumerate}
\item
Two multivector fields $U$ and $U'$ as above are $f$-related if and only if 
$i(U)\o f^* = f^*\o i(U'):\Om(N)\to \Om(M)$. 
\item
$U$ and $U'$ as above are $f$-related if and only if 
$\L(U)\o f^* = f^*\o \L(U'):\Om(N)\to \Om(M)$ holds.
\item
If $U_j$ and $U_j'$ are $f$-related for $j=1,2$ then also their Schouten brackets $[U_1,U_2]$ and 
$[U_1',U_2']$ are $f$-related. \hfil \qed
\end{enumerate}
\end{proposition}

\begin{lemma}\nmb.{2.6} Let $X\in \X(M)$ be a smooth vector field admitting a local flow $\Fl^X_t$; 
see \cite[32.13--32.17]{KrieglMichor97}. Then we have 
\begin{equation*}
\L_XU = \p_t|_0 (\Fl^X_t)^* U = [X,U].\hfill\qed
\end{equation*}
\end{lemma}

\section{The general Schouten bracket}\nmb0{3}

\subsection{General multivector fields and summable differential forms}\nmb.{3.1} For a convenient manifold the \emph{general multivector fields} of order $k$ are the smooth sections of the vector bundle $L^k_{\text{skew}}(T^*M, \mathbb R)\to M$. 
We denote these as follows: 
$$\on{MV}(M) = \sum_{k=0}^\infty \on{MV}^k(M) := \sum_{k=0}^\infty \Ga(L^k_{\text{skew}}(T^*M, \mathbb R))$$

A \emph{summable differential form} $\om$ on $M$ is a smooth section  of the bundle of skew 
symmetric tensors 
${\textstyle\bigwedge}^k_{\text{sum}}T^*M \subset \bar\otimes_\be^k T^*M = 
T^*M\bar\otimes_\be T^*M\bar\otimes_\be\dots\bar\otimes_\be T^*M\to M$, where $\bar\otimes_\be$ denotes the $c^\infty$-completed bornological tensor product which linearizes bounded bilinear mappings as used in \nmb!{1.3}.

Let us denote by $\Om^k_{\text{sum}}(M)$ the graded algebra of all summable differential 
forms. Note that exterior derivative $d:\Om^k(M)\to \Om^{k+1}(M)$ does not map 
$\Om^k_{\text{sum}}(M)$ into $\Om^{k+1}_{\text{sum}}(M)$; this fails even for $k=1$. Summability of a form is destroyed by the exterior derivative: For example, on a real Hilbert space  for  a bounded operator $A$ the 1-form $\om_x(X) = \langle Ax , X\rangle$ has exterior  derivative $d\om(X,Y)_x = \langle AX,Y\rangle - \langle AY,X\rangle- \langle Ax,[X,Y]\rangle$ is not summable if the operator $A-A^\top$ is not of trace class. 

Therefore we let 
$\Om^k_{\text{sum},d}(M)$ be the graded differential subalgebra of all summable forms 
$\om$ such that $d\om$ is again summable. Note that the latter condition is  a linear partial differential relation. 
By the condition in \nmb!{1.1} for the manifold $M$ we have:
\\
\emph{For each $\al\in T^*M$ there exists $f\in C^{\infty}(M)$ with $df_{\pi(\al)}=\al$.}
Consequently,
\begin{equation*}
\on{ev}_x\o\, d : \Om^k_{\text{sum,}d}(M) \to  \bigwedge^{k+1} T^*_xM \text{ is surjective for all }x\in M
\tag{1} \end{equation*}

The vector bundle $L^k_{\text{skew}}(T^*M, \mathbb R)\to M$ is the dual bundle of 
${\textstyle\bigwedge}^k_{\text{sum}}T^*M\to M$; we will denote the duality by (the dual space is always on the feft hand side)
\begin{gather*}
\langle\quad,\quad\rangle: L^k_{\text{skew}}(T^*M, \mathbb R)\x_M {\textstyle\bigwedge}^k_{\text{sum,}\be}T^*M \to \mathbb R
\\
\langle U,\ph_1\wedge \dots\wedge \ph_k\rangle = U(\ph_1,\dots,\ph_k)
\end{gather*} 
as well as its extension to spaces of sections.

For $\om\in \Om^k_{\text{sum}}(M)$ we consider the pointwise linear (i.e., vector bundle push-forward) mapping 
\begin{align*}
\mu(\om): \Om^\ell_{\text{sum}}(M)\to \Om^{\ell+k}_{\text{sum}}(M),\quad \mu(\om)\ph = \om\wedge \ph 
\end{align*}
and its pointwise dual 
\begin{align*}
\bar\io(\om)= \mu(\om)^*&: \on{MV}^{\ell+k}(M)\to \on{MV}^{\ell}(M),
\\
\langle U,\mu(\om)\ph\rangle &= \langle U,\om\wedge \ph\rangle = \langle \bar \io(\om)U,\ph\rangle
\end{align*}
For a decomposable k-form $\om=\ph_1\wedge \dots\wedge \ph_k$ we have 
\begin{align*}
\langle\bar \io(\ph_1\wedge \dots\wedge \ph_k)U,\ph_{k+1}\wedge \dots\wedge \ph_{k+\ell}\rangle &=  
\langle U,\ph_1\wedge \dots\wedge \ph_k\wedge \ph_{k+1}, \dots, \ph_{k+\ell})   \tag{2}
\\&
=U(\ph_1,\dots,\ph_{k+\ell})
\end{align*}
Similarly, for $U\in \on{MV}^u(M)$ we consider 
$$\bar\mu(U):\on{MV}^\ell(M)\to \on{MV}^{u+\ell}(M),\quad \bar\mu(U)V = U \wedge V$$
which is the dual of 
$i(U): \Om^{\ell+u}_{\text{sum}}(M)\to \Om^\ell_{\text{sum}}(M)$
which on decomposable $u+\ell$-forms is given by
\begin{align*}
i(U)&(\ph_1\wedge \dots\wedge \ph_{u+\ell}) =
\\&
= \frac1{u!\ell!}\sum_{\si\in\mathcal S_{k+\ell}}\on{sign}(\si) U(\ph_{\si(1)},\dots,\ph_{\si(u)}) \ph_{\si(u+1)}\wedge \dots\wedge \ph_{\si(k+\ell)}\,, \tag{3}
\end{align*}
since we have
\begin{align*}
\langle V&, i(U)(\ph_1\wedge \dots\wedge \ph_{u+v})\rangle = \langle U\wedge V,\ph_1\wedge \dots\wedge \ph_{u+v}\rangle
\\&
= (U\wedge V)(\ph_1,\dots,\ph_{u+v})
\\&
=\frac1{u!\,v!}\sum_{\si\in \mathcal S_{u+v}} \on{sign}(\si) U(\ph_{\si(1)},\dots, \ph_{\si(u)})
V(\ph_{\si(u+1)}, \dots, \ph_{\si(u+v)})
\\&
=\langle V,\frac1{u!\,v!}\sum_{\si\in \mathcal S_{u+v}} \on{sign}(\si) U(\ph_{\si(1)},\dots, \ph_{\si(u)})
\ph_{\si(u+1)}\wedge  \dots\wedge  \ph_{\si(u+v)}\big\rangle\,.
\end{align*}
\emph{\thetag{4} Formula \thetag{3} implies that $i(U)$ respects the $d$-stable subalgebra $\Om^k_{\text{sum},d}(M)$ so that 
\begin{equation*}
i(U): \Om^{\ell+u}_{\text{sum},d}(M)\to \Om^\ell_{\text{sum},d}(M)
\end{equation*}
}

We first have to redo Lemma \nmb!{2.2}
\begin{lemma}\nmb.{3.2}
Let $U$ be in $\on{MV}^u(M)$ and $V\in \on{MV}^v(M)$. Then we have:
\begin{enumerate}
\item 
$i(U):\Om_{\text{sum}}(M)\to \Om_{\text{sum}}(M)$ is a homogeneous bounded module 
homomorphism of degree $-u$.  It is a graded derivation of $\Om_{\text{sum}}(M)$ if and only if $u=1$. For $f\in C^{\infty}(M)$ we have $i(f)\om=f.\om$.
\item 
$i(U\wedge V)=i(V)\o i(U)$, thus the graded commutator vanishes: 
$$
[i(U),i(V)] = i(U)i(V) - (-1)^{uv}i(V)i(U)=0.
$$
\item For $\om\in\Om^1(M)$ and $\ps\in\Om_{\text{sum}}(M)$ we have 
\begin{gather*}
i(U)(\om\wedge \ps) = i(\bar\io(\om)U)\ps + (-1)^u\om\wedge i(U)\ps,
\quad
[i(U),\mu(\om)] = i(\bar\io(\om)U)
\end{gather*}
\end{enumerate}
\end{lemma}

\begin{proof}
\thetag{1} follows from the definition. 
\thetag{2} is the dual of $\bar \mu (U\wedge V)W= U\wedge V\wedge W = \bar\mu(U)\bar\mu(V)W$.
\thetag{3} For 1-forms $\ph_j$ we have
\begin{align*}
&[i(U),\mu(\ph_1)](\ph_2\wedge \dots\wedge \ph_k) = i(U)(\ph_1\wedge \dots\wedge \ph_{k}) -
\\&
- (-1)^{u} \ph_1\wedge i(U)(\ph_2\wedge \dots\wedge \ph_k)
\\&
= \tfrac1{u!(k-u)!}\sum_{\si\in\mathcal S_{k}}\on{sign}(\si) U(\ph_{\si(1)},\dots,\ph_{\si(u)}) \ph_{\si(u+1)}\wedge \dots\wedge \ph_{\si(k)}
\\&
- (-1)^{u} \ph_1\wedge \tfrac1{u!(k-u-1)!}\!\!\!\!\sum_{\si\in\mathcal S_{k-1}}\on{sign}(\si) U(\ph_{\si(2)},\dots,\ph_{\si(u+1)}) \ph_{\si(u+2)}\wedge \dots\wedge \ph_{\si(k)}
\\
&i(\bar\io(\ph_1)U)(\ph_2\wedge \dots\wedge \ph_k) =  
\\&
= \tfrac1{(u-1)!(k-u)!}\sum_{\si\in\mathcal S_{k-1}}\on{sign}(\si) (\bar\io(\ph_1)U)(\ph_{\si(2)},\dots,\ph_{\si(u)}) \ph_{\si(u+1)}\wedge \dots\wedge \ph_{\si(k)}
\\&
= \tfrac1{(u-1)!(k-u)!}\!\!\sum_{\si\in\mathcal S_{k-1}} \on{sign}(\si)  U (\ph_{1},\ph_{\si(2)},\dots,\ph_{\si(u)}) \ph_{\si(u+1)}\wedge \dots\wedge \ph_{\si(k)}
\\&\implies 
[i(U),\mu(\om)] = i(\bar\io(\om)U)  \qedhere
\end{align*}
\end{proof}

\subsection{Naturality}\nmb.{3.3}
Let $f:M\to N$ be a smooth mapping between convenient manifolds. 
Then we have a well defined pullback operation 
\begin{gather*}
f^*:\Om^k_{\text{sum}}(N)\to \Om^k_{\text{sum}}(M) \,,\quad
(f^*\om)_x = \big(\bigwedge^k(T_xf)^*\big)\om_{f(x)} 
\end{gather*}
which intertwines  via the `insertions' (which might not be injective in general) $\Om^k_{\text{sum}}(M)\to \Om^k(M)$ 
with the usual pullback operations. 

We say that 
$U\in \on{MV}^u(M)$ and $U'\in \on{MV}^u(N)$ are $f$-related, if for each $x\in M$ we have   
$L^u_{\text{skew}}(T_x^*f,\mathbb R).U_x = U'_{f(x)} $:
$$
\xymatrix@C=20mm{
L^u_{\text{skew}} (T^*M,\mathbb R) \ar[r]^{L^u_{\text{skew}} (T^*f,\mathbb R)} & L^u_{\text{skew}} (T^*N,\mathbb R)\\
M \ar[u]^{U} \ar[r]^{f} & N \ar[u]_{U'}
}
$$
Like in Proposition \nmb!{2.5} we have:
\\
\emph{Two multivector fields $U$ and $U'$ as above are $f$-related if and only if }
\begin{equation*}
i(U)\o f^* = f^*\o i(U'):\Om_{\text{sum}}(N)\to \Om_{\text{sum}}(M).
\end{equation*}

\subsection{\nmb.{3.4}The Schouten-Nijenhuis bracket for general multivector fields: Tulczyjew's Approach \cite{Tulczyjew74}}
We turn Theorem \nmb!{2.3} around and use \thetag{\nmb!{2.3}.7} as definition. 
For $\om\in \Om^{u+v-2}_{\text{sum,}d}(M) $
we put
\begin{equation*}
\langle [U,V], d\om \rangle = -\langle V, di(U)d\om\rangle  +(-1)^{(u-1)(v-1)}   \langle U,di(V)d\om \rangle\,.  \tag{1}
\end{equation*}
This remains valid if we insert a closed form in $\Om^{u+v-1}_{\text{sum,}d}(M)$ instead of $d\om$, since \thetag{1} is a local formula and locally each closed form is exact by the Lemma of Poincar\'e \cite[33.20]{KrieglMichor97}; the proof given there restricts to the subalgebra $\Om^{\bullet}_{\text{sum}}(M)$.
The right hand side of \thetag{1} is bounded linear in the variable $d\om$, hence \thetag{1} is uniquely defined, and it
readily extend as follows: Let $fd\om= f'd\om'$, then by \thetag{\nmb!{3.2}.3} we get
\begin{align*}
&f\langle U,di(V)d\om \rangle -(-1)^{(u-1)(v-1)} f\langle V, di(U)d\om\rangle =
\\&
= \langle U,di(V)(fd\om) \rangle -(-1)^{(u-1)(v-1)}\big( \langle V, di(U)(fd\om)\rangle 
\\&\qquad
- (-1)^{(u-1)v}\langle U\wedge V,d(fd\om)\rangle
\\&
= \langle U,di(V)(f'd\om') \rangle -(-1)^{(u-1)(v-1)}\big( \langle V, di(U)(f'd\om')\rangle 
\\&\qquad
- (-1)^{(u-1)v}\langle U\wedge V,d(f'd\om')\rangle
\\&
=f'\langle U,di(V)d\om' \rangle -(-1)^{(u-1)(v-1)} f'\langle V, di(U)d\om'\rangle.
\end{align*}
Therefore we also have
\begin{equation*}
\langle [U,V], fd\om \rangle = -f\langle V, di(U)d\om\rangle  +(-1)^{(u-1)(v-1)}   f\langle U,di(V)d\om \rangle \,. \tag{2}
\end{equation*}
Even on a smoothly paracompact manifold $M$ it is not obvious that 
\thetag{2} implies that $[U,V]$ is multivector field of order $u+v-1$. 
But \thetag{1} suffices for that. By naturality \nmb!{3.3} we assume that we are in a chart, or that $M$ is $c^\infty$-open in a convenient vector space: Then for constant 1-forms $\ph_i$ we compute \thetag{1}:  
\begin{align*}
&\langle [U,V], \ph_1\wedge \dots \ph_{u+v-1} \rangle = -\langle V, di(U)( \ph_1\wedge \dots \ph_{u+v-1})\rangle  +
\\&\qquad\qquad\qquad
+(-1)^{(u-1)(v-1)}   \langle U,di(V)( \ph_1\wedge \dots \ph_{u+v-1})\rangle 
\\&
= -\Big\langle V, d\Big(\frac1{u!(v-1)!}\sum_{\si\in\mathcal S_{u+v-1}} U(\ph_{\si(1)},\dots,\ph_{\si(u)}) \ph_{\si(u+1)}\wedge \dots \ph_{\si(u+v-1)}\Big)\Big\rangle+ \cdots
\\&
= - \frac1{u!(v-1)!}\sum_{\si\in\mathcal S_{u+v-1}} V\big(d(U(\ph_{\si(1)},\dots,\ph_{\si(u)})) \wedge \ph_{\si(u+1)}\wedge \dots \ph_{\si(u+v-1)}\big) + \cdots
\end{align*}
which is visibly a multivector field, since the $\ph_i$ are constant 1-forms; one may view  the $U(\ph_{\si(1)},\dots,\ph_{\si(u)})$ as coefficient functions.
Therefore 
\begin{equation*}
[\;,\;]: \on{MV}^u(M) \x \on{MV}^v(M) \to \on{MV}^{u+v-1}(M) \tag{3}
\end{equation*}
is a smooth (bounded) bilinear operator satisfying $[U,V]= -(-1)^{(u-1)(v-1)}[V,U]$.
It also satisfies
\begin{equation*}
\bar\io(df)[U,V] = [\bar\io(df)U,V] +(-1)^{u-1}[U,\bar\io(df)V]  \tag{4}
\end{equation*}
since by \thetag{\nmb!{3.2}.3} we have for $\om\in\Om^{u+v-3}_{\text{sum},d}(M)$
\begin{align*}
&\langle \bar\io(df)[U,V], d\om\rangle = \langle [U,V],(df\wedge d\om)\rangle
\\&
= - \langle V, di(U)(df\wedge d\om)\rangle  +(-1)^{(u-1)(v-1)}   \langle U,di(V)(df\wedge d\om) \rangle
\\&
=  -\langle V, di(\bar\io(df)U) d\om\rangle  +(-1)^u \langle V,df\wedge d i(U)d\om\rangle
\\&\quad
+(-1)^{(u-1)(v-1)}  \langle U,di(\bar\io(df)V)d\om \rangle +(-1)^{u(v-1)} \langle U, df\wedge di(V)d\om\rangle
\\&
\big\langle [\bar\io(df)U,V] +(-1)^{u-1} [U,\bar\io(df)V],d\om\big\rangle 
\\&
= -\langle V,di(\bar\io(df)U)d\om\rangle +(-1)^{u(v-1)} \langle \bar\io(df)U,di(V)d\om\rangle  
\\&\quad
+(-1)^{u-1}\big(-\langle \bar\io(df)V,di(U)d\om\rangle     +(-1)^{(u-1)v} \langle U,di(\bar\io(df)V)d\om\rangle\big)  
\end{align*}

\subsection{\nmb.{3.5}The general Schouten-Nijenhuis bracket, using Lie differentials}
By \thetag{\nmb!{3.1}.3} 
for $U\in \on{MV}^u(M)$ the {\em Lie differential operator}  
\begin{equation*}
\tag{1}   \L(U) := [i(U),d] = i(U)\o d - (-1)^u d\o i(U):\Om^k_{\text{sum},d}(M)\to \Om^{k-u+1}_{\text{sum},d}(M)
\end{equation*}
is well defined. 
It is a 
derivation if and only if $U$ is a vector field. We have $[\L(U),d]=0$ by the graded Jacobi 
identity of the graded commutator. We now generalise Theorem \nmb!{2.3} to this new situation:

\begin{theorem*}
Let $U\in\on{MV}^u(M)$, $V\in \on{MV}^v (M)$, $f\in C^\infty(M)$ and $\om\in \Om^{k}_{\text{sum,}d}(M)$.
Then we have:
\begin{align*}
\L(U\wedge V) &= i(V)\o \L(U) + (-1)^u\L(V)\o i(U)
\tag{2}\\
\L(X_1\wedge \dots\wedge X_u) &= \sum_j 
(-1)^{j-1}i(X_u)\cdots i(X_{j+1})\L(X_j)i(X_{j-1})\cdots i(X_1)
\tag{3}\\
[\L(U),i(V)] &= (-1)^{(u-1)(v-1)} i([U,V]) = -i([V,U])
\tag{4}\\
[\L(U),\L(V)] &= (-1)^{(u-1)(v-1)} \L([U,V]) = -\L([V,U])
\tag{5}
\\
 i([U,V]) &= -i(V)di(U) +(-1)^{(u-1)(v-1)}i(U)di(V)
\\&\quad
+(-1)^vdi(U\wedge V) +(-1)^{{u}}i(U\wedge V)d 
\tag{6} 
\\\text{The graded }&\text{Jacobi identity}\tag{7}
\\
[U,[V,W]] &= [[U,V],W] +(-1)^{u(v-1)}[V,[U,W]] 
\\
[U,V\wedge W] & = [U,V]\wedge W +(-1)^{(u-1)(v-1)} v\wedge [U,W] \tag{8}
\end{align*}
\end{theorem*}

Property \thetag{6} extends the definition \thetag{\nmb!{3.4}.1} to the more general situation. It corresponds to \cite[3.3]{Tulczyjew74}.

\begin{proof} 
 \thetag{2} has the same proof as \thetag{\nmb!{2.3}.2}. From it we see that also \thetag{3} is true, so the Schouten-Nijenhuis bracket restricts to one treated in Section \nmb!{2} for summable multivector fields.
\\
\thetag{4} 
We prove that  $[\L(U),i(V)]\om = -i([V,U])\om$ is valid for
$\om\in\Om^k_{\text{sum},d}(M)$ by induction on $k$. For this it suffices to show that 
\begin{align*}
[[\L(U),i(V)] +i([V,U]), \mu(df)] = 0
\end{align*}
By the graded Jacobi identiy for the graded commutator and by using  \thetag{\nmb!{3.2}.3} and 
$[d,\mu(df)]=0$ we have first
\begin{align*}
[\L(U),\mu(df)] &= [[i(U),d],\mu(df)] = [i(U),[d,\mu(df)]] - (-1)^u[d,[i(U),\mu(df)]] \tag{a}
\\&
= 0 -(-1)^u [d, i(\bar\io(df)U)] = -\L(\bar\io(df)U)
\end{align*}
and then
\begin{align*}
[[&\L(U),i(V)],\mu(df)] = [\L(U),[i(V),\mu(df)]] - (-1)^{(u-1)v} [i(V),[\L(U),\mu(df)]] 
\\&
= [\L(U),i(\bar\io(df)V)] +(-1)^{v-1} [\L(\bar\io(df)U),i(V)] 
\\
[i&([V,U]),\mu(df)] = i\big(\bar\io(df)[V,U]\big) = i\big([\bar\io(df)V,U] + (-1)^{v-1}[V,\bar\io(df)U]\big)\text{ by \thetag{\nmb!{3.4}.4}}
\\
[[&\L(U),i(V)] +i([V,U]), \mu(df)] = [\L(U),i(\bar\io(df)V)] +  i\big([\bar\io(df)V,U]\big)
\\&
+ (-1)^{v-1} \big([\L(\bar\io(df)U),i(V)] + i([V,\bar\io(df)U])\big) = 0
\end{align*}
by induction on $u+v$.
\\
\thetag{5} We prove  that $\big([\L(U),\L(V)] +\L([V,U])\big)\om =0$ for all $\om\in\Om^k_{\text{sum},d}(M)$ by induction on $u+v+k$.
If suffices to  show $[[\L(U),\L(V)] +\L([V,U],\mu(df)]$ = 0 for all $f$.
\begin{align*}
[[\L&(U),\L(V)],\mu(df)] = [\L(U),[\L(V),\mu(df)]] - (-1)^{(u-1)(v-1)}[\L(V),[\L(U),\mu(df)]]
\\& 
=[\L(U),-\L(\bar\io(df)V)] - (-1)^{(u-1)(v-1)}[\L(V),-\L(\bar\io(df)U)] \quad\text{ by \thetag{a}}
\\& 
=\L([\bar\io(df)V,U]) - (-1)^{(u-1)(v-1)}[\L([\bar\io(df)U,V])\quad\text{ by induction on }u+v
\\& 
=\L\big([\bar\io(df)V,U] - (-1)^{v}[V,\bar\io(df)U]\big) = \L\big(\bar\io(df)[V,U]\big) \quad\text{ by \thetag{\nmb!{3.4}.4}}
\\&
=[-\L([V,U]),\mu(df)]
\end{align*}
\thetag{6} follows from \thetag{4}, the graded Jacobi identity for the graded commutator, and \thetag{\nmb!{3.2}.2} as follows:
\begin{align*}
i([U,V]) &= - [\L(V),i(U)] = -[[i(V),d],i(U)] 
\\&
=-i(V)di(U) +(-1)^{(u-1)(v-1)}i(U)di(V)
\\&\quad+(-1)^vdi(V)i(U) +(-1)^{{(v-1)u}}i(U)i(V)d 
\end{align*}
\thetag{7} Since $U\mapsto \L(U)$ is injective by the assumptions \nmb!{1.1} on $M$, this follows from \thetag{5} and the graded Jacobi identity for the graded commutator:
\begin{align*}
\L([U,[V,W]]) &= -[\L([V,W]),\L(U)] = [[\L(W),\L(V)],\L(U)] 
\\&
= [\L(W),[\L(V),\L(U)]] -(-1)^{(w-1)(v-1)}[\L(V),[\L(W),\L(U)]]
\\&
= -[\L(W),\L([U,V])] +(-1)^{(w-1)(v-1)}[\L(V),\L([U,W])]
\\&
= +\L([[U,V],W]) -(-1)^{(w-1)(v-1)}\L([[U,W],V])
\\&
= +\L([[U,V],W]) +(-1)^{u(v-1)}\L([V,[U,W]])
\end{align*}
\end{proof}

\begin{theorem}\nmb.{3.6}
Let $P\in \on{MV}^2(M)$ or in $\Ga(\bigwedge^2 TM)$. Then the skew symmetric product 
$\{f,g\}:= \langle P, df\wedge dg \rangle \in C^\infty(M,\mathbb R)$ (where $f,g\in C^\infty(M,\mathbb R)$) 
satisfies the Jacobi identity if and only if $[P,P]=0$
\end{theorem}

This well known result is the infinite dimensional version of \cite[1.4]{Michor87a}.

\begin{proof}
For $f,g\in C^\infty(M,\mathbb R)$ we have 
$$
\{f,g\}:= \langle P,df\wedge dg \rangle = \langle \bar\io(df)P, dg\rangle = 
\langle  -\bar\io(dg)[f,P],1\rangle = [g,[f,P]]\,.
$$
Now a straightforward computation using the graded Jacobi identity and skew symmetry of the 
Schouten-Nijenhuis bracket gives:
$$
[h,[g,[f,[P,P]]]] = -2\Big(\{f,\{g,h\}\}+\{g,\{h,f\}\}+\{h,\{f,g\}\}\Big)\,.
$$
Since $[h,[g,[f,[P,P]]]] = \langle  df\wedge dg\wedge dh,[P,P]\rangle$, the result follows. 
\end{proof}

\def\cprime{$'$}

\end{document}